\documentclass[12pt]{amsart}
\usepackage{amssymb}

\usepackage[english]{babel}
\usepackage[autostyle]{csquotes}

\theoremstyle{plain}
\newtheorem{theorem}{Theorem}[section]

\newtheorem{proposition}[theorem]{Proposition}

\theoremstyle{definition}
\newtheorem{definition}[theorem]{Definition}
\newtheorem{remark}[theorem]{Remark}

\theoremstyle{remark}

\newcommand{\norm}[1]{\lVert#1\rVert}

\newcommand{\Bignorm}[1]{\Bigl\lVert#1\Bigr\rVert}

\newcommand{\C}{\mathbb{C}}
\newcommand{\N}{\mathbb{N}}
\newcommand{\B}{\mathcal{B}}
\newcommand{\veps}{\varepsilon}

\newcommand{\Z}{\mathcal{Z}}

\newcommand\restr[2]{{
  \left.\kern-\nulldelimiterspace 
  #1 
  \vphantom{\big|} 
  \right|_{#2} 
  }}

\DeclareMathOperator{\Ker}{Ker}

\begin{document}
\baselineskip 18pt
\title{Invariant subspace problem for rank-one perturbations: the quantitative version}

\author[A.~Tcaciuc]{Adi Tcaciuc}
\address[A. Tcaciuc]{Mathematics and Statistics Department,
   Grant MacEwan University, Edmonton, Alberta, Canada T5J
   P2P, Canada}
\email{tcaciuca@macewan.ca}


\begin{abstract}
We show that for any bounded operator $T$ acting on infinite dimensional, complex Banach space, and for any $\varepsilon>0$,  there exists an operator $F$ of rank at most one and norm smaller than $\varepsilon$ such that $T+F$ has an invariant subspace of infinite dimension and codimension. A version of this result was proved in \cite{T19} under additional spectral conditions for $T$ or $T^*$. This solves in full generality the quantitative version of the invariant subspace problem for rank-one perturbations.
\end{abstract}
\maketitle


\section{Introduction}\label{intro}

This paper is a continuation of the work in \cite{T19}, which showed the existence of invariant subspaces for rank-one perturbations for general Banach spaces. A partial solution was given in \cite{T19} to the quantitative version of this question.  Here we solve in full generality this quantitative version as well.

The Invariant Subspace Problem, asking whether every bounded operator acting on a separable complex Banach space has a non-trivial closed invariant subspace, is one of the most famous problem in Operator Theory. It is still open for the most important case, of a separable Hilbert space. There is a vast literature on the Invariant Subspace Problem and, for brevity, we refer the readers to the  monographs by Radjavi and Rosenthal \cite{RR03} and by Chalendar and Partington \cite{CP11}, for a more comprehensive review of this topic, as well as for more recent results and approaches.

A related problem, the existence of invariant subspaces for perturbations of bounded operators has been studied for a long time, in particular in the context of separable Hilbert spaces. For example, Brown and Pearcy~\cite{BP71} showed that for any $T\in\mathcal{B}(\mathcal{H})$, where $\mathcal{H}$ is an infinite-dimensional separable Hilbert space, and for any $\varepsilon > 0$, there exists a compact operator $K$ with norm at most $\varepsilon$ such that $T+K$ has an invariant subspace of infinite dimension and codimension. As an immediate consequence of Voiculescu's ~\cite{V76} famous non-commutative Weyl-von Neumann Theorem it follows that $K$ as above can be chosen  such that $T+K$ has a \emph{reducing} subspace of infinite dimension and codimension, that is, a subspace that is invariant for both $T+K$ and its adjoint.

A new approach to the existence of invariant subspaces for finite-rank perturbations was introduced by Androulakis, Popov, Tcaciuc, and Troitsky \cite{APTT09}. It is easy to see that, given a bounded operator $T\in\mathcal{B}(X)$, where $X$ is an infinite dimensional, complex Banach space, any finite dimensional or finite-codimensional subspace is invariant under some suitable finite rank perturbation. Thus, in this context of searching for non-trivial closed invariant subspaces for a finite rank perturbations, a ''non-trivial'' subspace is a subspace of infinite dimension and codimension. Such a subspace will be henceforth called a \emph{half-space}. A half-space that is invariant for some finite rank perturbation of $T$ is called \emph{almost-invariant} for $T$ (see Section \ref{def} for more details). In \cite{APTT09} the authors showed that certain weighted shifts admit rank-one perturbations that have invariant half-spaces. For reflexive Banach spaces, Popov and Tcaciuc \cite{PT13}proved that every bounded operator admits a rank-one perturbation that has an invariant half-space. In the same paper, the authors show the existence of suitable perturbations that also have small norms, provided certain spectral conditions are satisfied; in particular this gave a substantial improvement over the aforementioned result of Brown and Pearcy. For general Banach spaces, partial solutions were given in \cite{SW14}, \cite{SW16}, and \cite{TW17}. The question was also studied for algebras of operators in \cite{P10},\cite{MPR13}, and \cite{SW16}. Please see the Introduction of \cite{T19} for a brief overview.

The question was solved for general Banach spaces in \cite{T19}, by showing that every bounded operator acting on an arbitrary (separable) complex Banach space, admits a rank-one perturbation that has an invariant-half-space. Similarly with the solution in the reflexive case in \cite{PT13}, the existence of rank-one perturbations that have small norms was also proved, provided again that certain spectral conditions hold. In this paper we solve this quantitative question in full generality,  removing these limiting spectral conditions,  and prove the following theorem:
\begin{theorem}\label{maintheorem}
Let $X$ be a complex Banach space and $T\in\B(X)$. Then for any $\veps>0$, there exists $F\in\mathcal{B}(X)$ of rank at most one and $\norm{F}\leq\veps$ such that $T+F$ has an invariant half-space.
\end{theorem}
In Section \ref{def} we establish the notations, and introduce the definitions and tools employed later in the paper. In particular, we review the relevant results in this direction and present the assumptions we can make in view of these results. Section \ref{mainsection} contains the proof of Theorem \ref{maintheorem}; for the sake of clarity we isolate a important part of this proof in a separate result, Proposition \ref{mainprop}.

\section{Definitions and Preliminaries}\label{def}

We are going to use the same definitions and notations as in \cite{T19}; we briefly review them here as well. For a Banach space $X$, we denote by $\mathcal{B}(X)$ the algebra of all bounded operators on $X$. When  $T\in{\mathcal B}(X)$, we write $\sigma(T)$, $\sigma_p(T)$, $\rho(T)$ and $\partial\sigma(T)$ for the spectrum
of~$T$,  point spectrum of $T$,  the resolvent set of~$T$ and the topological boundary of the spectrum, respectively. The closed span of a set  $\{x_n\}_n$ of vectors in $X$ is denoted by $[x_n]$. A sequence $(x_n)_{n=1}^{\infty}$ in $X$ is called a \emph{basic sequence} if any $x\in[x_n]$ can be written uniquely as $x=\sum_{n=1}^{\infty} a_n x_n$, where the convergence is in norm (see \cite[section 1.a]{LT77} for background on Schauder bases and basic sequences).

The following definition was introduced in \cite{APTT09}, providing an equivalent formulation of the existence of finite rank perturbations having invariant half-spaces.
\begin{definition}
If $T\in\mathcal{B}(X)$ and $Y$ is a subspace of $X$, we say that $Y$ is \emph{almost-invariant} for $T$ if there exists a finite dimensional subspace $E$ of $X$ such that $TY\subseteq Y+E$. The smallest dimension of such an $E$ is called the \emph{defect} of $Y$ for $T$
\end{definition}

Indeed, it is not hard to prove (see Proposition 1.3 in \cite{APTT09}) that a half-space $Y$ is almost-invariant with defect $k$ for a bounded operator $T$, if and only if there exists a rank $k$ operator $F$ such that $Y$ is invariant for $T+F$.  The new approach introduced in \cite{APTT09} essentially studies techniques of constructing these almost-invariant half-spaces, rather than directly the suitable finite-rank perturbations.

As we mentioned in the introduction, the main result in \cite{T19} is that \emph{for any separable, complex, Banach spaces, and any $T\in\mathcal{B}(X)$, there exists a bounded operator $F$ of rank at most one, such that $T+F$ has an invariant half-space.}  There are well-known examples of bounded operators that have only finite dimensional or infinite dimensional invariant subspaces (e.g. certain weighted shifts on $\mathcal{H})$. Therefore, when it comes to existence of invariant half-spaces for \emph{all} operators in $\mathcal{B}(X)$ , this result is the most one can hope for.

The following quantitative version was also proved in \cite{T19}.

\begin{theorem} \cite{T19} \label{oldmain}
Let $X$ be a separable Banach spaces, and let $T\in\B(X)$ such that $\partial\sigma(T)\setminus\sigma_p(T)\neq\emptyset$ or $\partial\sigma(T^*)\setminus\sigma_p(T^*)\neq\emptyset$. Then for any $\veps>0$, there exists $F\in\mathcal{B}(X)$ of rank at most one and $\norm{F}\leq\veps$ such that $T+F$ has an invariant half-space. If the spectral condition does not hold, we can still find a finite-rank $F$ with $\norm{F}\leq\veps$, but not necessarily rank-one.
\end{theorem}

In this paper we obtain rank at most one (and small norm), even when the spectral condition does not hold. Therefore, we may assume that any value in $\partial\sigma(T)=\partial\sigma(T^*)$ is an eigenvalue for both $T$ and $T^*$. Note that for any $\lambda\in\partial\sigma(T)$, $\mu\in\partial\sigma(T^*)$ with $\lambda\neq\mu$, and for corresponding eigenvectors $x$ and $x^*$ we have that $x^*(x)=0$. Hence, when $\partial\sigma(T)=\partial\sigma(T^*)$ we can construct an \emph{invariant} half-space for $T$, spanned by countably many eigenvectors for $T$ (see Proof of Theorem 2.7 in \cite{PT13}) for details.

Therefore, we can further assume that $\partial\sigma(T)=\partial\sigma(T^*)$ is finite, therefore  $\sigma(T^*)=\sigma(T)=\partial\sigma(T)$ is finite and consists only of eigenvalues. In this case we can assume without loss of generality that $\sigma(T)$ is a singleton. Indeed, one of the finitely many Riesz projections associated to each of the eigenvalues in $\sigma(T)$ must have infinite dimensional range, and this infinite dimensional range is a $T$-invariant closed infinite dimensional subspace, but not necessarily a half-space. The restriction of $T$ to this infinite dimensional subspace has singleton as its spectrum, and if we can find a suitable, small in norm, rank-one perturbation of this restriction, then clearly we can do the same for $T$ (please see the details in the proof of Theorem 3.3 in \cite{T19}). Hence we can assume $\sigma(T)=\{\lambda\}$.  Note that if the conclusion of Theorem \ref{maintheorem} holds for $T$ then it hold for any $T-\lambda I$. Therefore,  by replacing $T$ with $T-\lambda I$, we can assume that $\sigma(T)=\{0\}$, that is, $T$ is quasinilpotet.

To summarize, in order to prove the conclusion of Theorem \ref{maintheorem}, we can assume that $T$ (hence $T^*$ as well) is quasinilpotent, and $\sigma_p(T)=\sigma_p(T^*)=\{0\}$.

For the proof we are going to use the following  $w^*$-analogue of the Bessaga-Pelczynski selection principle, due to  Johnson and Rosenthal (see Theorem III.1 and Remark III.1 in \cite{JR72}). For a shorter proof see the more recent paper of Gonz\'{a}lez and Martinez-Abej\'{o}n \cite{GM12}. Recall that a sequence $(x_n)$ in a Banach space is called \emph{semi-normalized} if $0<\inf\|x_n\|\leq\sup\|x_n\|<\infty$.

\begin{theorem}\cite{JR72}\cite{GM12} \label{JR}
  If $(x^*_n)$ is a semi-normalized, $w^*$-null, sequence in a dual Banach space $X^{*}$, then there exists a basic subsequence $(y^*_n)$ of $(x^*_n)$, and a bounded sequence $(y_n)$ in $X$ such that $y^*_i(y_j)=\delta_{ij}$ for all $1\leq i,j<\infty$.
\end{theorem}

\section{Rank one perturbations with small norm}\label{mainsection}

We first prove the following proposition, as a first step towards the general case.

\begin{proposition}\label{mainprop}
  Let $X$ be a separable Banach space, $T\in\mathcal{B}(X)$ a quasinilpotent operator, and $\mathcal{Z}$ an infinite dimensional, separable,   $T^*$-invariant subspace of $X^*$ such that $\restr{T^*}{\Z}$ has dense range. Then for any $\veps>0$, there exists $F\in\mathcal{B}(X)$ of rank at most one and $\norm{F}\leq\veps$ such that $T+F$ has an invariant half-space.
\end{proposition}

\begin{proof}
Denote by $S:=T^*_{|\mathcal{Z}}:\Z\to\Z$ the restriction of $T^*$ to $\Z$,  and by $j:X\to X^{**}$ the canonical embedding of $X$ in the double dual $X^{**}$. Since $T$ is quasinilpotent, so is $T^*$, hence also $S$ and $S^*:\Z^*\to\Z^*$. Also, since $S$ has dense range, it follows that $S^*$ is injective, so $0\notin\sigma_{p}(S^*)$. Pick $x\in X$, $\norm{x}=1$, such that the restriction of $j(x)$ to $\Z$, $\restr{j(x)}{\Z}\neq 0$, and denote by $x^{**}:=\restr{j(x)}{\Z}\in\Z^*$. Consider the \emph{local resolvent} of $S^*$ at $x^{**}$,  $g:\C\setminus\{0\}\to \Z^{*}$ defined as:
$$
g(z):=(zI-S^{*})^{-1}(x^{**})=\sum_{n=0}^{\infty}\frac{S^{*n} x^{**}}{z^{n+1}}
$$

We have that $g$ is analytic on $C\setminus\{0\}$, and $z=0$ is either a pole or an essential singularity for $g$. Therefore $g$ is unbounded near $z=0$, so we can find a sequence $\lambda_n\to 0$ such that $\norm{g(\lambda_n)}\to\infty$. Denote:
$$
h_n^{**}:=(\lambda_n I-S^{*})^{-1}(x^{**})\in\Z^{*} \textrm{ and } h_n:=(\lambda_n I-T)^{-1}(x)\in X.
$$

It is routine to check that $h_n^{**}=\restr{j(h_n)}{\Z}$. Denote by $y_n^{**}:=h_n^{**}/\norm{h_n^{**}}$, and notice that
\begin{equation}\label{eq1}
  S^*y_n^{**}=\lambda_ny_n^{**}-\frac{x^{**}}{\norm{h_n^{**}}}
\end{equation}

We have from Banach-Alaoglu that $B_{\Z^*}$,  unit ball of $\Z^*$, is $w^*$-compact (in the $w^*$-topology of $Z^*$), and since $\Z$ is separable it follows that $B_{\Z^*}$ is also $w^*$-metrizable. Without loss of generality assume that $y_n^{**}\stackrel{w^*}{\longrightarrow}y^{**}\in\Z^{*}$. Taking $w^*$-limits in (\ref{eq1}), and using the fact that $\lambda_n\to 0$ and $\norm{h_n^{**}}\to\infty$, easy calculations show that $S^*y^{**}=0$. From the injectivity of $S^*$ it now follows that $y^{**}=0$.

From Theorem \ref{JR}, we can assume that $(y_n^{**})$ is a basic sequence in $\Z^{*}$, and that there exists $(y_n^{*})$ a bounded sequence in $\Z$ such that $y_i^{**}(y_j^{*})=\delta_{ij}$, for all $1\leq i, j<\infty$. Clearly both $(y_n^{**})$ and $(y_n^{*})$ are linearly independent, and note that for any $n\neq k$ we have that
\begin{equation}\label{eq2}
y_n^{*}(h_k)=j(h_k)(y_n^{*})=h_k^{**}(y_n^{*})=0
\end{equation}

and

\begin{equation}\label{eq3}
  y_k^*(h_k)=j(h_k)(y_k^*)=h_k^{**}(y_k^*)=\norm{h_k^{**}}y_k^{**}(y_k^*)=\norm{h_k^{**}}
\end{equation}

From (\ref{eq2}) it follows easily that $(h_n)\subseteq X$ are linearly independent, hence $[h_n]$ is an infinite dimensional subspace of $X$. On the other hand, for any $k\in\N$, we have that $y_{2k+1}^*([h_{2n}])=0$, and since $(y_{2k+1}^{*})$ are linearly independent it follows that $[h_{2n}]$ is an infinite dimensional subspace of $X$. Therefore, by eventually passing to a subsequence, we may assume that $[h_n]$ is a half-space.

Fix $\veps>0$ and let $M$ such that $\norm{y_n^{*}}<M$. Since $\norm{h_n^{**}}\to\infty$, by eventually passing to a subsequence we may assume that $\sum_{n=1}^{\infty}1/\norm{h_n^{**}}<\veps/M$. For any $k\in\N$ we have
$$
\Bignorm{\sum_{n=1}^{k}\frac{1}{\norm{h_n^{**}}}y_n^*}\leq\sum_{n=1}^{k}\frac{1}{\norm{h_n^{**}}}\norm{y_n^*}
\leq M\sum_{n=1}^{\infty}\frac{1}{\norm{h_n^{**}}}\leq M\cdot \frac{\veps}{M}=\veps.
$$

Therefore $f^*:=\sum_{n=1}^{\infty}\frac{1}{\norm{h_n^{**}}}y_n^*\in\Z$ is well defined and $\norm{f^*}<\veps$. Consider the rank-one operator $F:=x\otimes f^{*}\in\B(X)$. Then $\norm{F}<\veps$ and for any $k\in\N$ we have
$$
F(h_k)=f^*(h_k)x=\left(\sum_{n=1}^{\infty}\frac{1}{\norm{h_n^{**}}}y_n^*(h_k)\right)x=\frac{1}{\norm{h_k^{**}}}y_k^*(h_k)x=x
$$
and hence

$$
(T+F)h_k=Th_k+Fh_k=\lambda_k h_k-x+x=\lambda_k h_k
$$

This shows that $[h_n]$ is an invariant half-space for $T+F$.
\end{proof}

\begin{remark}
The main difficulties to overcome in Proposition \ref{mainprop} are for the situation when $X$ is not reflexive,  the proof can be simplified in the reflexive case.
\end{remark}


We are now ready to prove Theorem \ref{maintheorem}

\begin{proof}[Proof of Theorem \ref{maintheorem}]

From Theorem \ref{oldmain} and the discussion following it, we may assume that $T$ is quasinilpotent and $\sigma_p(T)=\sigma_p(T^*)=\{0\}$. We may also assume that $\Ker(T)$ is finite dimensional, as otherwise any infinite codimensional subspace of $\Ker(T)$ is an invariant half-space for $T$.

For a non-zero vector $x\in X$, denote by $\mathcal{C}_x$ the closed span of the orbit of $x$ under $T$, that is $\mathcal{C}_x:=[T^n x]_{n\geq 0}$. Clearly, when  $\mathcal{C}_x$ is finite dimensional we can find a polynomials $p_x$ such that $p_x(T)x=0$. If $\mathcal{C}_x$ is finite dimensional for all $x\in X$, then $T$ is said to be locally algebraic and it follows from Kaplansky Lemma (see Lemma 14 in \cite{K71}) that it is actually algebraic. That is, there exists a polynomial $p$ such that $p(T)=0$. In this case, since $T$ is quasinilpotent, it must be nilpotent, so $\ker(T)$ is infinite dimensional, a contradiction with our initial assumptions. Therefore, there exists a non-zero $x\in X$ such that $\mathcal{C}_x$ is infinite dimensional. If the restriction of $T$ to $\mathcal{C}_x$, $\restr{T}{\mathcal{C}_x}:\mathcal{C}_x\to\mathcal{C}_x$, satisfies the conclusion of the theorem, clearly so does $T$; therefore without loss of generality we may assume that $\mathcal{C}_x=X$.

First note that if $\mathcal{C}_x=X$ then $\dim\Ker(T^*)=0$ or $\dim\Ker(T^*)=1$. Indeed, suppose $y^*\in\ker{T^*}$. Then, for any natural number $n\geq 0$ we have that $(T^*y^*)(T^n x)=0$, therefore $y^*(T^{n}x)=0$ for any $n\geq 1$. If $x\in[T^n x]_{n\geq 1}$, then $[T^n x]_{n\geq 1}=\mathcal{C}_x=X$, and it follows that $y^*=0$. Otherwise, $[T^n x]_{n\geq 1}$ is a $1$-codimensional subspace of $X$, and from the fact that  $y^*(T^{n}x)=0$ for all $y^*\in\Ker(T^*)$ we have that $\dim\Ker(T^*)=1$.
From our initial assumption that $\sigma_p(T^*)=\{0\}$, we cannot have that $\dim\Ker(T^*)=0$, therefore we must have that $\dim\Ker(T^*)=1$.

Next, we are going to show that for any $n\in\N$, $\dim\Ker(T^{*n+1})-\dim\Ker(T^{*n})\leq 1$. Indeed, note that for any $n\in\N$ we have the short exact sequence
$$
0\longrightarrow\Ker(T^{*n})\stackrel{i}{\longrightarrow}\Ker(T^{*n+1})\stackrel{T^{*n}}{\longrightarrow} T^{*n}(\Ker(T^{*n+1}))\longrightarrow 0.
$$

Hence, also taking into account that $T^{*n}(\Ker(T^{*n+1}))\subseteq\Ker{T^*}$, it follows that $\dim\Ker(T^{*n+1})\leq\dim\Ker(T^{*n})+\dim\Ker(T^*)$. Since $\dim\Ker(T^*)=1$, we conclude that for any $n\in\N$, $\dim\Ker(T^{*n+1})-\dim\Ker(T^{*n})\leq 1$.

For any $n\in\mathbb{N}$, denote by $Y_n=\overline{Range(T^n)}$, and put $Y_0:=X$. Clearly each $Y_n$ is a closed invariant subspace of $X$, $Y_{n+1}=\overline{TY_n}$,  and $X\supseteq Y_1\supseteq Y_2\supseteq\dots.$ we can also  assume each $Y_j$ is infinite dimensional; indeed, otherwise, if $j$ is the smallest index for which $Y_j$ is finite dimensional, then any half-space of $Y_{j-1}$ containing $Y_j$ is an invariant half-space for $T$.

We consider two cases.

Case I: $\exists n\in\mathbb{N}$ such that $\Ker(T^{*n})=\Ker(T^{*n+1})$

In this situation we have that $Y_n=Y_{n+1}$; indeed, this is easy to see as for any $y^*\in X^*$ we have that $y^*\in\Ker(T^{*n})$ if and only if $y^*(Y_n)=0$. Therefore $S:=\restr{T}Y_n:Y_n \to Y_n$ has dense range, and hence  $S^*: Y_n^*\to Y_n^*$ is injective, and also quasinilpotent. From Theorem \ref{oldmain} applied to $S$ we conclude that $S$, hence also $T$, satisfies the conclusion of the Theorem, and we are done.

Case II: $\forall n\in\mathbb{N}$, $\dim\Ker(T^{*n+1})-\dim\Ker(T^{*n})=1$. In this situation we are going to show that $T$ satisfies the hypotheses of Proposition \ref{mainprop}.

Note that we have a strict inclusion $\Ker(T^*)\subset\Ker(T^{*2})\subset\Ker(T^{*3})\dots$, where for any $n\in\mathbb{N}$, $\dim\Ker(T^{*n})=n$. Pick $x_1^*\in \Ker(T)$, and for any $n\in\mathbb{N}$, $n\geq 2$,  pick $x_n^*\in \Ker(T^{*n})\setminus\Ker(T^{*n-1})$. Clearly $(x_n^*)$ are linearly independent, and put $\mathcal{Z}:=[x_n^*]\subseteq X^*$. We have that $\mathcal{Z}$ is a separable subspace of $X^*$, and to finish the proof remains to show that $\mathcal{Z}$ is $T^*$-invariant and $\restr{T^*}{\Z}$ has dense range.

We have that for any $n\in\mathbb{N}$, $\Ker(T^{*n})=[x_j^*]_{1\leq j\leq n}\subseteq\mathcal{Z}$ and also that $T^* x_n^*\in\Ker(T^{*n-1})$. Hence $T^*\mathcal{Z}\subseteq\mathcal{Z}$, so $\mathcal{Z}$ is $T^*$-invariant. To show that $\restr{T^*}{\Z}$ has dense range, we first claim  that for any $n\in\mathbb{N}$, $x_n^*\in T^*\mathcal{Z}$. Since $T^*x_2^*\in\Ker(T^*)=[x_1^*]$, it follows that $x_1^*\in T^*\mathcal{Z}$. Assume by induction that $x_j^*\in T^*\mathcal{Z}$ for any $1\leq j< n$ and we will show that $x_n^*\in T^*\mathcal{Z}$. Since $T^*x_{n+1}^*\in\Ker{T^{*n}}$, we can write
$$
T^*x_{n+1}^*=\alpha_1 x_1^*+\alpha_2 x_2^*+\dots+ \alpha_n x_n^*
$$

Since $x_1^*, x_2^*, \dots, x_{n-1}^*$, and $T^*x_{n+1}^*$ are all in $T^*\mathcal{Z}$, suffices to show that $\alpha_n\neq 0$. If $\alpha_n=0$, then we have that $T^*x_{n+1}^*\in[x_j]_{1\leq j\leq n-1}=\Ker(T^{*n-1})$, therefore $x_n^*\in\Ker(T^{*n})$. This is a contradiction with the choice of $x_n^*$, and the claim is proved.

Hence $\mathcal{Z}\subseteq T^*\mathcal{Z}\subseteq\overline{T^*\mathcal{Z}}$. Therefore $\overline{T^*\mathcal{Z}}=\mathcal{Z}$, that is, $\restr{T^*}{\Z}$ has dense range,  and we obtain the conclusion of the Theorem by applying Proposition \ref{mainprop}.

This finishes the proof.

\end{proof}

\vskip .3cm

\emph{Acknowledgments.}  This research was supported in part by the Natural Sciences and Engineering Research Council of Canada, grant number 2019-07097.


\begin{thebibliography}{999}
\bibitem[APTT09]{APTT09}
George Androulakis, Alexey I Popov, Adi Tcaciuc and Vladimir G. Troitsky.
\newblock {\it Almost Invariant Half-spaces of Operators on Banach Spaces}
\newblock    Integr.Equ.Oper.Theory 65 (2009), 473--484.


 \bibitem[BP71]{BP71}
  A. Brown, C. Pearcy,
  \newblock \emph{Compact restrictions of operators}
  \newblock  Acta Sci. Math. (Szeged) {\bf 32} (1971), 271-282.

  \bibitem[CP11]{CP11}
  I. Chalendar, J.R.~Partington,
  \newblock \emph{Modern approaches to the invariant-subspace problem}
  \newblock  Cambridge Tracts in Mathematics, 188. Cambridge University Press, Cambridge, 2011.

\bibitem[GM12]{GM12}
Manuel Gonz\'{a}lez, Antonio Martinez-Abej\'{o}n
\newblock {\it On basic sequences in dual Banach spaces},
\newblock  JMAA  {\bf 395(2)}  (2012), 813-814.


\bibitem[JR72]{JR72}
W.B. Johnson, H.P. Rosenthal
\newblock {\it On $w^{*}$-basic sequences and their aplications to the study of Banach spaces},
\newblock  Studia Math.  {\bf 43}  (1972), 77-92.

\bibitem[K71]{K71}
I. Kaplansky
\newblock{Infinite Abelian Groups}
\newblock Univ. of Michigan Press, Ann Arbour, 1971.


  \bibitem[LT77]{LT77}
  J.~Lindenstrauss, L.~Tzafriri,
  \newblock \emph{Classical Banach spaces. I. Sequence spaces}.
  \newblock Ergebnisse der Mathematik und ihrer Grenzgebiete,
  Vol. 92. Springer-Verlag, Berlin-New York, 1977.


\bibitem[MPR13]{MPR13}
Laurent W. Marcoux, Alexey I. Popov and Heydar Radjavi,
\newblock {\it On Almost-invariant Subspaces and Approximate Commutation},
\newblock  J. Funct. Anal.  {\bf 264(4)}  (2013), 1088–1111.

\bibitem[P10]{P10}
Alexey I. Popov
\newblock{\it Almost invariant half-spaces of algebras of operators}
\newblock Integr.Equ.Oper.Theory {\bf 67(2)} (2010), 247-256

\bibitem[PT13]{PT13}
Alexey I. Popov and Adi Tcaciuc.
\newblock {\it Every Operator has Almost-invariant Subspaces},
\newblock  J. Funct. Anal.  {\bf 265(2)}  (2013), 257–265.

  \bibitem[RR03]{RR03}
  H.~Radjavi, P.~Rosenthal,
  \newblock \emph{Invariant subspaces, Second edition}.
  \newblock Dover Publications, Inc., Mineola, NY, 2003.


\bibitem[SW14]{SW14}
Gleb Sirotkin and Ben Wallis.
\newblock {\it The structure of almost-invariant half-spaces for some operators}
\newblock  J. Funct. Anal. {\bf 267}(2014), 2298-2312.

\bibitem[SW16]{SW16}
Gleb Sirotkin and Ben Wallis.
\newblock {\it Almost-invariant and essentially-invariant halfspaces}
\newblock Linear Algebra Appl. {\bf 507} (2016), 399-413


\bibitem[TW17]{TW17}
Adi Tcaciuc, Ben Wallis.
\newblock{\it Controlling almost-invariant halfspaces in both real and complex setting}
\newblock Integr.Equ.Oper.Theory {\bf 87(1)} (2017), 117-137.


\bibitem[T19]{T19}
Adi Tcaciuc
\newblock {\it The invariant subspace problem for rank-one perturbations}
\newblock Duke Math. J. {\bf 168} (2019), 1539-1550.


 \bibitem[V76]{V76}
   D. V. Voiculescu,
  \newblock \emph{A non-commutative Weyl-von Neumann theorem}
  \newblock Rev. Roumaine Math. Pures Appl. 21 (1976), 97-113.



\end{thebibliography}
\end{document}